\documentclass[12pt]{article}%
\usepackage[intlimits]{amsmath}
\usepackage{amssymb}
\usepackage{latexsym}
\usepackage{tikz}
\usepackage{anysize}
\usepackage{subfigure}
\usepackage[T1]{fontenc}
\usepackage[sc]{mathpazo}
\usepackage{color}
\usepackage[colorlinks]{hyperref}
\usepackage{amsfonts}
\usepackage{graphicx}%
\definecolor {refcol}{RGB}{40,0,255}
\hypersetup{colorlinks=true,allcolors=refcol}
\makeatletter
\newfont{\footsc}{cmcsc10 at 8truept}
\newfont{\footbf}{cmbx10 at 8truept}
\newfont{\footrm}{cmr10 at 10truept}
\newtheorem{theorem}{Theorem}

\newtheorem{corollary}[theorem]{Corollary}

\newtheorem{definition}[theorem]{Definition}

\newtheorem{lemma}[theorem]{Lemma}

\newtheorem{proposition}[theorem]{Proposition}

\newenvironment{proof}[1][Proof]{\noindent{\textbf {#1}  }}  {\hfill$\Box$\bigskip}
\begin{document}

\title{\textbf{On the }$A_{\alpha}$\textbf{-spectra of trees}}
\author{Vladimir Nikiforov\thanks{Department of Mathematical Sciences, University of
Memphis, Memphis TN 38152, USA.} , Germain Past\'{e}n\thanks{Department of
Mathematics, Universidad Cat\'{o}lica del Norte, Antofagasta, Chile} , Oscar
Rojo\footnotemark[2] , and Ricardo L. Soto\footnotemark[2]}
\date{}
\maketitle

\begin{abstract}
Let $G$ be a graph with adjacency matrix $A(G)$ and let $D(G)$ be the diagonal
matrix of the degrees of $G$. For every real $\alpha\in\left[  0,1\right]  ,$
define the matrix $A_{\alpha}\left(  G\right)  $ as
\[
A_{\alpha}\left(  G\right)  =\alpha D\left(  G\right)  +(1-\alpha)A\left(
G\right)
\]
where $0\leq\alpha\leq1$.

This paper gives several results about the $A_{\alpha}$-matrices of trees. In
particular, it is shown that if $T_{\Delta}$ is a tree of maximal degree
$\Delta,$ then the spectral radius of $A_{\alpha}(T_{\Delta})$ satisfies the
tight inequality
\[
\rho(A_{\alpha}(T_{\Delta}))<\alpha\Delta+2(1-\alpha)\sqrt{\Delta-1}.
\]
This bound extends previous bounds of Godsil, Lov\'{a}sz, and Stevanovi\'{c}.
The proof is based on some new results about the $A_{\alpha}$-matrices of
Bethe trees and generalized Bethe trees.

In addition, several bounds on the spectral radius of $A_{\alpha}$ of general
graphs are proved, implying tight bounds for paths and Bethe trees.

\end{abstract}



\section{Introduction}

Let $G$ be a graph with adjacency matrix $A(G)$, and let $D\left(  G\right)  $
be the diagonal matrix of its vertex degrees. In \cite{Nik16}, it was proposed
to study the family of matrices $A_{\alpha}(G)$ defined for any real
$\alpha\in\left[  0,1\right]  $ as
\[
A_{\alpha}(G)=\alpha D(G)+(1-\alpha)A(G).
\]
Since $A_{0}\left(  G\right)  =A\left(  G\right)  $ and $2A_{1/2}\left(
G\right)  =Q\left(  G\right)  $, where $Q\left(  G\right)  $ is the signless
Laplacian of $G,$ the matrices $A_{a}$ can underpin a unified theory of
$A\left(  G\right)  $ and $Q\left(  G\right)  .$

The primary purpose of this paper is to study the $A_{\alpha}$-matrices of
trees. Our first goal is to give a tight upper bound on the spectral radius
$\rho\left(  A_{\alpha}\left(  T_{\Delta}\right)  \right)  $, where
$T_{\Delta}$ is a tree with maximal degree $\Delta.$ Let us recall that for
the adjacency matrix Godsil \cite{God84} gave the tight bound
\[
\rho\left(  A\left(  T_{\Delta}\right)  \right)  <2\sqrt{\Delta-1}.
\]
The crucial idea of the proof, which Godsil attributes to I. Gutman, is the
fact that a tree of maximal degree $\Delta$ can be embedded into a
sufficiently large Bethe tree of degree $\Delta.$ To estimate the spectral
radius of such Bethe trees, Godsil applied an intricate result from
\cite{HeLi72}. Later, Lov\'{a}sz solved the same problem, and found precisely
the spectral radius of Bethe trees (see \cite{Lov93}, Problems 5 and 14).
Independently, Stevanovi\'{c} \cite{Ste03} also proposed a self-contained
calculation, and, in addition, proved the tight bound
\[
\rho\left(  Q\left(  T_{\Delta}\right)  \right)  <\Delta+2\sqrt{\Delta-1},
\]
which was stated for the spectral radius $\rho\left(  L\left(  T_{\Delta
}\right)  \right)  $ of the Laplacian of $T_{\Delta},$ but since trees are
bipartite graphs, we have $\rho\left(  Q\left(  T_{\Delta}\right)  \right)
=\rho\left(  L\left(  T_{\Delta}\right)  \right)  $.

In the following theorem, we extend the results of Godsil, Lov\'{a}sz, and
Stevanovi\'{c} to the whole family $A_{\alpha}\left(  T_{\Delta}\right)  $:

\begin{theorem}
\label{t1}If $T_{\Delta}$ is a tree of maximal degree $\Delta$ and $\alpha
\in\left[  0,1\right]  ,$ then
\[
\rho\left(  A_{\alpha}\left(  T_{\Delta}\right)  \right)  <\alpha
\Delta+2(1-\alpha)\sqrt{\Delta-1}.
\]
This bound is tight.
\end{theorem}

To prove Theorem \ref{t1}, we calculate the spectra of certain Bethe trees,
providing, in fact, more than is needed for the proof of Theorem \ref{t1}:
namely, in Section \ref{gbt}, we introduce generalized Bethe trees and give a
reduction procedure for calculating their $A_{\alpha}$-spectra, thereby
extending the main results of \cite{RoSo05}.\medskip

Our next result, proved in Section \ref{ton}, is an absolute upper bound on
$\rho(A_{\alpha}(T))$ of a tree of order $n$. For $A\left(  G\right)  $ such
result has been proved in \cite{LoPe73}, and for $Q\left(  G\right)  $ in
\cite{Che02}.

\begin{theorem}
\label{t2}If $T$ is a tree of order $n$ and $\alpha\in\left[  0,1\right]  ,$
then
\[
\rho(A_{\alpha}(T))\leq\frac{\alpha n+\sqrt{\alpha^{2}n^{2}+4\left(
n-1\right)  \left(  1-2\alpha\right)  }}{2}.
\]
Equality holds if and only if $T$ is the star $K_{1,n-1}$.
\end{theorem}

In the opposite direction, we show that $P_{n}$, the path of order $n$, has
minimal spectral radius among all connected graphs of order $n$:

\begin{theorem}
\label{t3}If $G$ is a connected graph of order $n$ and $\alpha\in\left[
0,1\right]  ,$ then
\[
\rho(A_{\alpha}(G))\geq\rho(A_{\alpha}(P_{n})).
\]
Equality holds if and only if $G=P_{n}$.
\end{theorem}

Results similar to Theorem \ref{t3} were given for $A\left(  G\right)  $ in
\cite{CoSi57} and for $Q\left(  G\right)  $ in \cite{Che02}, but the proof
presented in Section \ref{cg} turns out to be more involved.

In the final Section \ref{Sb}, we give upper and lower bounds for the spectral
radius of the $A_{\alpha}$-matrices of arbitrary graphs; in particular, we
deduce tight bounds for paths and Bethe trees.\medskip

\section{\label{gbt}Generalized Bethe trees}

Given a rooted graph, define the\emph{ level of a vertex} to be equal to its
distance to the root vertex increased by one. A \emph{generalized Bethe tree}
is a rooted tree in which vertices at the same level have the same degree.

Throughout this paper, $B_{k}$ denotes a generalized Bethe tree on $k$ levels.
Let $\left[  k\right]  $ denote the set $\left\{  1,\ldots,k\right\}
.$\ Given a $B_{k}$ and an integer $j\in\left[  k\right]  ,$ write $n_{k-j+1}$
for the number of vertices at level $j$ and $d_{k-j+1}$ for their degree. In
particular, $d_{1}=1$ and $n_{k}=1$.

Further, any $j\in\left[  k-1\right]  $, let $m_{j}=n_{j}/n_{j+1}$. Then, for
any $j\in\left[  k-2\right]  $, we see that
\begin{equation}
n_{j}=(d_{j+1}-1)n_{j+1},\label{nj}%
\end{equation}
and, in particular,
\begin{equation}
n_{k-1}=d_{k}=m_{k-1}.\label{nk}%
\end{equation}
It is worth pointing out that $m_{1},\ldots,m_{k-1}$ are always positive
integers, and that $n_{1}\geq\cdots\geq n_{k}.$

We label the vertices of $B_{k}$ with the numbers $1,\ldots,n$, starting at
the last level $k$ and ending at the root vertex; at each level the vertices
are labeled from left to right, as illustrated in Fig. 1.\medskip

\begin{figure}[ptb]%
\begin{align*}
\begin{tikzpicture} \tikzstyle{every node}=[draw,circle,fill=black,minimum size=3pt, inner sep=0pt] \draw (0,0) node (67) [label=above:$67$] {} (6,-1.5) node (66) [label=left:$66$] {} (0,-1.5) node (65) [label=left:$65$] {} (-6,-1.5) node (64) [label=left:$64$] {} (8,-3) node (63) [label=left:$63$] {} (6,-3) node (62) [label=left:$62$] {} (4,-3) node (61) [label=left:$61$] {} (2,-3)node (60) [label=left:$60$] {} (0,-3) node (59) [label=left:$59$] {} (-2,-3)node (58) [label=left:$58$] {} (-4,-3) node (57) [label=left:$57$] {} (-6,-3) node (56) [label=left:$56$] {} (-8,-3) node (55) [label=left:$55$] {} (8.5,-4.5) node (54) [label=left:$54$] {} (7.5,-4.5) node (53) [label=left:$53$] {} (6.5,-4.5) node (52) [label=left:$52$] {} (5.5,-4.5) node (51) [label=left:$51$] {} (4.5,-4.5) node (50) [label=left:$50$] {} (3.5,-4.5) node (49) [label=left:$49$] {} (2.5,-4.5)node (48) [label=left:$48$] {} (1.5,-4.5)node (47) [label=left:$47$] {} (0.5,-4.5) node (46) [label=left:$46$] {} (-0.5,-4.5) node (45) [label=left:$45$] {} (-1.5,-4.5)node (44) [label=left:$44$] {} (-2.5,-4.5)node (43) [label=left:$43$] {} (-3.5,-4.5) node (42) [label=left:$42$] {} (-4.5,-4.5) node (41) [label=left:$41$] {} (-5.5,-4.5) node (40) [label=left:$40$] {} (-6.5,-4.5) node (39) [label=left:$39$] {} (-7.5,-4.5) node (38) [label=left:$38$] {} (-8.5,-4.5) node (37) [label=left:$37$] {} (0.25,-6) node (19) [label=below:$19$] {} (0.75,-6) node (20) [label=below:$20$] {} (1.25,-6) node (21) [label=below:$21$] {} (1.75,-6) node (22) [label=below:$22$] {} (2.25,-6) node (23) [label=below:$23$] {} (2.75,-6) node (24) [label=below:$24$] {} (3.25,-6) node (25)[label=below:$25$] {} (3.75,-6) node (26) [label=below:$26$] {} (4.25,-6) node (27) [label=below:$27$] {} (4.75,-6) node (28) [label=below:$28$] {} (5.25,-6) node (29) [label=below:$29$] {} (5.75,-6) node (30) [label=below:$30$] {} (6.25,-6) node (31) [label=below:$31$] {} (6.75,-6) node (32) [label=below:$32$] {} (7.25,-6) node (33) [label=below:$33$] {} (7.75,-6) node (34) [label=below:$34$] {} (8.25,-6) node (35) [label=below:$35$] {} (8.75,-6) node (36) [label=below:$36$] {} (-0.25,-6) node (18) [label=below:$18$] {} (-0.75,-6) node (17) [label=below:$17$] {} (-1.25,-6) node (16) [label=below:$16$] {} (-1.75,-6) node (15) [label=below:$15$] {} (-2.25,-6) node (14) [label=below:$14$] {} (-2.75,-6) node (13) [label=below:$13$] {} (-3.25,-6) node (12)[label=below:$12$] {} (-3.75,-6) node (11) [label=below:$11$] {} (-4.25,-6) node (10) [label=below:$10$] {} (-4.75,-6) node (9) [label=below:$9$] {} (-5.25,-6) node (8) [label=below:$8$] {} (-5.75,-6) node (7) [label=below:$7$] {} (-6.25,-6) node (6) [label=below:$6$] {} (-6.75,-6) node (5) [label=below:$5$] {} (-7.25,-6) node (4) [label=below:$4$] {} (-7.75,-6) node (3) [label=below:$3$] {} (-8.25,-6) node (2) [label=below:$2$] {} (-8.75,-6) node (1) [label=below:$1$] {}; \draw (1)--(37); \draw (2)--(37); \draw (3)--(38); \draw (4)--(38); \draw (5)--(39); \draw (6)--(39); \draw (7)--(40); \draw (8)--(40); \draw (9)--(41); \draw (10)--(41); \draw (11)--(42); \draw (12)--(42); \draw (13)--(43);\draw (14)--(43);\draw (15)--(44);\draw (16)--(44);\draw (17)--(45);\draw (18)--(45);\draw (19)--(46);\draw (20)--(46); \draw (21)--(47);\draw (22)--(47);\draw (23)--(48);\draw (24)--(48);\draw (25)--(49);\draw (26)--(49);\draw (27)--(50);\draw (28)--(50);\draw (29)--(51);\draw (30)--(51);\draw (31)--(52);\draw (32)--(52);\draw (33)--(53);\draw (34)--(53);\draw (35)--(54);\draw (36)--(54); \draw (37)--(55);\draw (38)--(55);\draw (39)--(56);\draw (40)--(56);\draw (41)--(57);\draw (42)--(57);\draw (43)--(58);\draw (44)--(58);\draw (45)--(59);\draw (46)--(59);\draw (47)--(60);\draw (48)--(60);\draw (49)--(61);\draw (50)--(61);\draw (51)--(62);\draw (52)--(62);\draw (53)--(63);\draw (54)--(63); \draw (55)--(64);\draw (56)--(64);\draw (57)--(64);\draw (58)--(65);\draw (59)--(65);\draw (60)--(65);\draw (61)--(66);\draw (62)--(66);\draw (63)--(66); \draw (64)--(67);\draw (65)--(67);\draw (66)--(67); \end{tikzpicture}
\end{align*}
\caption{Labeling a generalized Bethe tree}%
\end{figure}
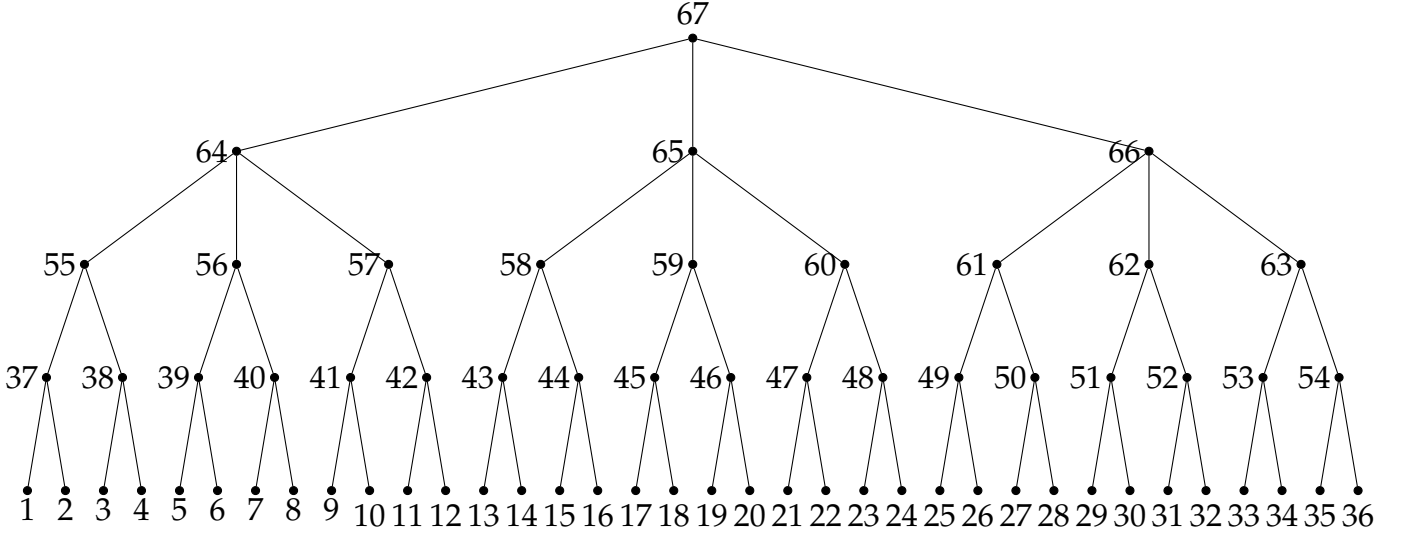

Recall that the Kronecker product $A\otimes B$ of two matrices $A=\left(
a_{i,j}\right)  $ and $B=\left(  b_{i,j}\right)  $ of sizes $m\times m$ and
$n\times n$, is an $mn\times mn$ matrix defined as $A\otimes B=\left(
a_{i,j}B\right)  .$

Two basic properties of $A\otimes B$ are the identities%
\[
\left(  A\otimes B\right)  ^{T}=A^{T}\otimes B^{T}%
\]
and%
\[
\left(  A\otimes B\right)  \left(  C\otimes D\right)  =\left(  AC\otimes
BD\right)  ,
\]
which hold for any matrices of appropriate sizes.\medskip

We write $I_{m}$ for the identity matrix of order $m$ and $\mathbf{j}_{m}$ for
the column $m$-vector of ones$.\medskip$

Set $\beta=1-\alpha,$ and assume that $B_{k}$ is a generalized Bethe tree
labeled as described above. It is not hard to see that the matrix $A_{\alpha
}(B_{k})$ can be represented as a symmetric block tridiagonal matrix
\[
A_{\alpha}(B_{k})=\left[
\begin{array}
[c]{cccccc}%
\alpha I_{n_{1}} & \beta I_{n_{2}}\otimes\mathbf{j}_{m_{1}} & 0 &  &  & 0\\
\beta I_{n_{2}}\otimes\mathbf{j}_{m_{1}}^{T} & \alpha d_{2}I_{n_{2}} & \beta
I_{n_{3}}\otimes\mathbf{j}_{m_{2}} &  &  & \\
&  & \ddots & \ddots & \ddots & \\
&  &  & \beta I_{n_{k-1}}\otimes\mathbf{j}_{m_{k-2}}^{T} & \alpha
d_{k-1}I_{n_{k-1}} & \beta\mathbf{j}_{m_{k-1}}\\
0 &  &  & 0 & \beta\mathbf{j}_{m_{k-1}}^{T} & \alpha d_{k}%
\end{array}
\right]  .
\]

In the next subsection we use this form to calculate the characteristic
polynomial of $A_{\alpha}(B_{k}).$

\subsection{The $A_{\alpha}$-spectra of $B_{k}$}

Given a generalized Bethe tree $B_{k}$, define the polynomials $P_{0}\left(
\lambda\right)  ,\ldots,P_{k}\left(  \lambda\right)  $ as follows:

\begin{definition}
\label{defP}Let%
\[
P_{0}\left(  \lambda\right)  =1,\;P_{1}\left(  \lambda\right)  =\lambda
-\alpha,
\]
and
\[
P_{j}(\lambda)=(\lambda-\alpha d_{j})P_{j-1}(\lambda)-\beta^{2}m_{j-1}%
P_{j-2}(\lambda)
\]
for $j=2,\ldots,k.$
\end{definition}

The polynomials $P_{0}\left(  \lambda\right)  ,\ldots,P_{k}\left(
\lambda\right)  $ can be used to express the characteristic polynomial of
$A_{\alpha}(B_{k})$, as shown in the following theorem:

\begin{theorem}
\label{poly}The characteristic polynomial $\phi\left(  \lambda\right)  $ of
$A_{\alpha}(B_{k})$ satisfies%
\begin{equation}
\phi\left(  \lambda\right)  =P_{k}(\lambda)\prod\limits_{j=1}^{k-1}%
P_{j}(\lambda)^{n_{j}-n_{j+1}}. \label{pA}%
\end{equation}

\end{theorem}

\begin{proof}
Write $\left\vert A\right\vert $ for the determinant of a square matrix $A.$
To prove (\ref{pA}), we shall reduce $\phi\left(  \lambda\right)  =|\lambda
I-A_{\alpha}(B_{k})|$ to the determinant of an upper triangular matrix. For a
start, note that
\[
\phi\left(  \lambda\right)  =\left\vert
\begin{array}
[c]{cccccc}%
P_{1}(\lambda)I_{n_{1}} & -\beta I_{n_{2}}\otimes\mathbf{j}_{m_{1}} & 0 &  &
& 0\\
-\beta I_{n_{2}}\otimes\mathbf{j}_{m_{1}}^{T} & (\lambda-\alpha d_{2}%
)I_{n_{2}} & -\beta I_{n_{3}}\otimes\mathbf{j}_{m_{2}} &  &  & \\
&  & \ddots & \ddots & \ddots & \\
&  &  & -\beta I_{n_{k-1}}\otimes\mathbf{j}_{m_{k-2}}^{T} & (\lambda-\alpha
d_{k-1})I_{n_{k-1}} & -\beta\mathbf{j}_{m_{k-1}}\\
0 &  &  & 0 & -\beta\mathbf{j}_{m_{k-1}}^{T} & \lambda-\alpha d_{k}%
\end{array}
\right\vert .
\]
Let $\lambda\in\mathbb{R}$ be such that $P_{j}\left(  \lambda\right)  \neq0$
for any $j\in\left[  k-1\right]  $, and for any $j\in\left[  k-1\right]  $,
set $P_{j}:=P_{j}\left(  \lambda\right)  .$

Multiplying the first row by $\frac{\beta}{P_{1}}I_{n_{2}}\otimes
\mathbf{j}_{m_{1}}^{T}$ and adding it to the second row, we obtain
\[
\phi\left(  \lambda\right)  =\left\vert
\begin{array}
[c]{ccccc}%
P_{1}I_{n_{1}} & -\beta I_{n_{2}}\otimes\mathbf{j}_{m_{1}} & 0 &  & 0\\
& (\lambda-\alpha d_{2}-\frac{\beta^{2}m_{1}}{P_{1}})I_{n_{2}} & -\beta
I_{n_{3}}\otimes\mathbf{j}_{m_{2}} &  & \\
&  & \ddots & \ddots & \ddots\\
&  & -\beta I_{n_{k-1}}\otimes\mathbf{j}_{m_{k-2}}^{T} & (\lambda-\alpha
d_{k-1})I_{n_{k-1}} & -\beta\mathbf{j}_{m_{k-1}}\\
0 &  & 0 & -\beta\mathbf{j}_{m_{k-1}}^{T} & \lambda-\alpha d_{k}%
\end{array}
\right\vert .
\]
Since
\[
\lambda-\alpha d_{2}-\frac{\beta^{2}m_{1}}{P_{1}}=\frac{(\lambda-\alpha
d_{2})P_{1}-\beta^{2}m_{1}P_{0}}{P_{1}}=\frac{P_{2}}{P_{1}},
\]
we find that
\[
\phi\left(  \lambda\right)  =\left\vert
\begin{array}
[c]{cccccc}%
P_{1}I_{n_{1}} & -\beta I_{n_{2}}\otimes\mathbf{j}_{m_{1}} & 0 &  &  & 0\\
& \frac{P_{2}}{P_{1}}I_{n_{2}} & -\beta I_{n_{3}}\otimes\mathbf{j}_{m_{2}} &
&  & \\
& -\beta I_{n_{3}}\otimes\mathbf{j}_{m_{2}}^{T} & (\lambda-\alpha
d_{3})I_{n_{3}} & \ddots &  & \\
&  & \ddots & \ddots & \ddots & \\
&  &  & -\beta I_{n_{k-1}}\otimes\mathbf{j}_{m_{k-2}}^{T} & (\lambda-\alpha
d_{k-1})I_{n_{k-1}} & -\beta\mathbf{j}_{m_{k-1}}\\
0 &  &  & 0 & -\beta\mathbf{j}_{m_{k-1}}^{T} & \lambda-\alpha d_{k}%
\end{array}
\right\vert .
\]
Next, multiply the second row by $\beta\frac{P_{1}}{P_{2}}I_{n_{3}}%
\otimes\mathbf{j}_{m_{2}}^{T}$ and add it to the third row. Using the
definition of $P_{3}$, we find that
\[
\phi\left(  \lambda\right)  =\left\vert
\begin{array}
[c]{cccccc}%
P_{1}I_{n_{1}} & -\beta I_{n_{2}}\otimes\mathbf{j}_{m_{1}} & 0 &  &  & 0\\
& \frac{P_{2}}{P_{1}}I_{n_{2}} & -\beta I_{n_{3}}\otimes\mathbf{j}_{m_{2}} &
&  & \\
&  & \frac{P_{3}}{P_{2}}I_{n_{3}} & \ddots &  & \\
&  & \ddots & \ddots & \ddots & \\
&  &  & -\beta I_{n_{k-1}}\otimes\mathbf{j}_{m_{k-2}}^{T} & (\lambda-\alpha
d_{k-1})I_{n_{k-1}} & -\beta\mathbf{j}_{m_{k-1}}\\
0 &  &  & 0 & -\beta\mathbf{j}_{m_{k-1}}^{T} & \lambda-\alpha d_{k}%
\end{array}
\right\vert .
\]
Continuing with this procedure, finally we multiply the $(k-1)$th row by
$\beta\frac{P_{k-2}}{P_{k-1}}\mathbf{j}_{m_{k-1}}^{T}$ and add it to the last
row, thus getting
\[
\phi\left(  \lambda\right)  =\left\vert
\begin{array}
[c]{cccccc}%
P_{1}I_{n_{1}} & -\beta I_{n_{2}}\otimes\mathbf{j}_{m_{1}} & 0 &  &  & 0\\
& \frac{P_{2}}{P_{1}}I_{n_{2}} & -\beta I_{n_{3}}\otimes\mathbf{j}_{m_{2}} &
&  & \\
&  & \ddots & \ddots & \ddots & \\
&  &  &  & \frac{P_{k-1}}{P_{k-2}}I_{n_{k-1}} & -\beta\mathbf{j}_{m_{k-1}}\\
0 &  &  &  & 0 & \frac{P_{k}}{P_{k-1}}%
\end{array}
\right\vert .
\]
Therefore,
\begin{align*}
\phi\left(  \lambda\right)   &  =P_{1}^{n_{1}}(\frac{P_{2}}{P_{1}})^{n_{2}%
}(\frac{P_{3}}{P_{2}})^{n_{3}}\ldots(\frac{P_{k-1}}{P_{k-2}})^{n_{k-1}}%
\frac{P_{k}}{P_{k-1}}\\
&  =P_{1}(\lambda)^{n_{1}-n_{2}}P_{2}(\lambda)^{n_{2}-n_{3}}\ldots
P_{k-1}(\lambda)^{n_{k-1}-1}P_{k}(\lambda).
\end{align*}
Thus, equality (\ref{pA}) is proved whenever $P_{j}\left(  \lambda\right)
\neq0$ for all $j\in\left[  k-1\right]  $. Since the polynomials
$P_{1}(\lambda),\ldots,$ $P_{k-1}(\lambda)$ have finitely many roots, the two
sides of (\ref{pA}) are equal for infinitely many values of $\lambda$; hence,
they are identical, completing the proof of Theorem \ref{poly}.
\end{proof}

\bigskip

\begin{definition}
\label{T} For $j=1,2,\ldots,k-1$, let $T_{j}$ be the $j\times j$ leading
principal submatrix of the $k\times k$ symmetric tridiagonal matrix
\begin{equation}
T_{k}=\left[
\begin{array}
[c]{ccccc}%
\alpha & \beta\sqrt{d_{2}-1} & 0 &  & 0\\
\beta\sqrt{d_{2}-1} & \alpha d_{2} & \ddots &  & \\
& \ddots & \ddots & \beta\sqrt{d_{k-1}-1} & \\
&  & \beta\sqrt{d_{k-1}-1} & \alpha d_{k-1} & \beta\sqrt{d_{k}}\\
0 &  & 0 & \beta\sqrt{d_{k}} & \alpha d_{k}%
\end{array}
\right]  , \label{mat}%
\end{equation}
where $\beta=1-\alpha$.
\end{definition}

Since $d_{s}>1$ for all $s=2,3,....,j$, each matrix $T_{j}$ has nonzero
codiagonal entries and it is known that its eigenvalues are simple.

Using the well known three-term recursion formula for the characteristic
polynomials of the leading principal submatrices of a symmetric tridiagonal
matrix and the formulae (\ref{nj}) and (\ref{nk}), one can easily prove the
following assertion:

\begin{lemma}
\label{TP}Let $B_{k}$ be a generalized Bethe tree, and $\alpha\in\left[
0,1\right)  .$ If the matrices $T_{1},\ldots,T_{k}$ are defined as in
Definition \ref{T}, then
\[
\left\vert \lambda I-T_{j}\right\vert =P_{j}(\lambda)
\]
for any $j\in\left[  k\right]  .$
\end{lemma}

Theorem \ref{poly}, Lemma \ref{TP}, and the interlacing property for the
eigenvalues of Hermitian matrices yield the following summary statement:

\begin{theorem}
\label{spectrum}Let $B_{k}$ be a generalized Bethe tree, and $\alpha\in\left[
0,1\right)  $. If the matrices $T_{1},\ldots,T_{k}$ are defined as in
Definition \ref{T}, then:

(1) The spectrum of $A_{\alpha}(B_{k})$ is the multiset union%
\begin{equation}
\mathrm{Spec}(T_{1})\cup\cdots\cup\mathrm{Spec}(T_{k})\text{;} \label{spec}%
\end{equation}

(2) The multiplicity of each eigenvalue of $T_{j}$ as an eigenvalue of
$A_{\alpha}(B_{k})$ is $n_{j}-n_{j+1}$ if \ $1\leq j\leq k-1$, and is $1$ if
$j=k$. If some eigenvalues obtained in different matrices are equal, their
multiplicities are added together;

(3) The largest eigenvalue of $T_{k}$ is the largest eigenvalue of $A_{\alpha
}(B_{k})$.
\end{theorem}

\subsection{Proof of Theorem \ref{t1}}

A \emph{Bethe tree} $B(d,k)$ is a rooted tree of $k$ levels in which:

- the root has degree $d$;

- the vertices at level $j\ \left(  2\leq j\leq k-1\right)  $ have degree
$d+1$;

- the vertices at level $k$ have degree equal to $1$ (pendant vertices).
\medskip

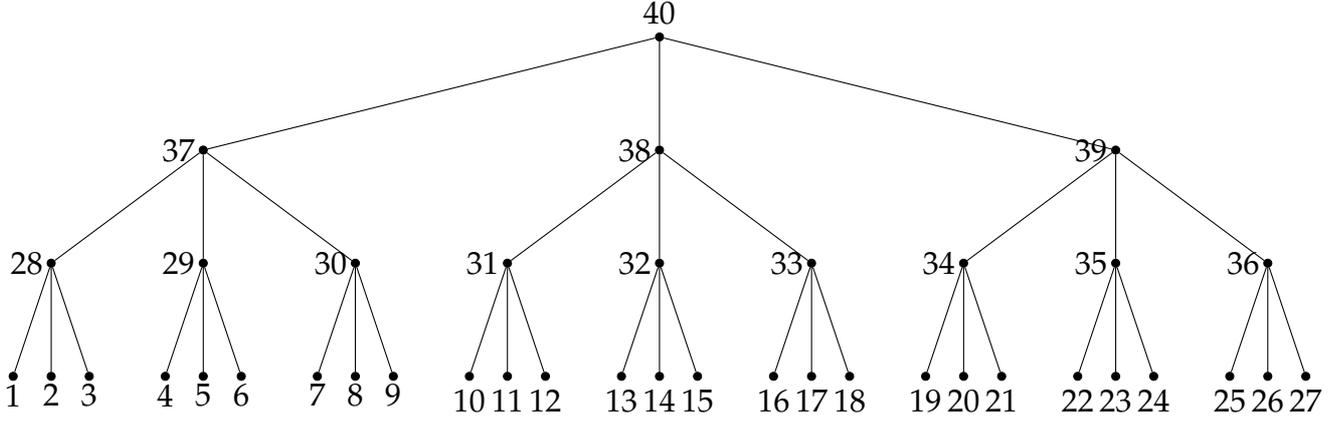
\begin{figure}[ptb]
\begin{tikzpicture}
\tikzstyle{every node}=[draw,circle,fill=black,minimum size=3pt,
inner sep=0pt]
\draw (0,0) node (40) [label=above:$40$] {}
(6,-1.5) node (39) [label=left:$39$] {}
(0,-1.5) node (38) [label=left:$38$] {}
(-6,-1.5) node (37) [label=left:$37$] {}
(8,-3) node (36) [label=left:$36$] {}
(6,-3) node (35) [label=left:$35$] {}
(4,-3) node (34) [label=left:$34$] {}
(2,-3)node (33) [label=left:$33$] {}
(0,-3) node (32) [label=left:$32$] {}
(-2,-3)node (31) [label=left:$31$] {}
(-4,-3) node (30) [label=left:$30$] {}
(-6,-3) node (29) [label=left:$29$] {}
(-8,-3) node (28) [label=left:$28$] {}
(8.5,-4.5) node (27) [label=below:$27$] {}
(8,-4.5) node (26) [label=below:$26$] {}
(7.5,-4.5) node (25) [label=below:$25$] {}
(6.5,-4.5) node (24) [label=below:$24$] {}
(6,-4.5) node (23) [label=below:$23$] {}
(5.5,-4.5) node (22) [label=below:$22$] {}
(4.5,-4.5) node (21) [label=below:$21$] {}
(4,-4.5) node (20) [label=below:$20$] {}
(3.5,-4.5) node (19) [label=below:$19$] {}
(2.5,-4.5)node (18) [label=below:$18$] {}
(2,-4.5) node (17) [label=below:$17$] {}
(1.5,-4.5)node (16) [label=below:$16$] {}
(0.5,-4.5) node (15) [label=below:$15$] {}
(0,-4.5) node (14) [label=below:$14$] {}
(-0.5,-4.5) node (13) [label=below:$13$] {}
(-1.5,-4.5)node (12) [label=below:$12$] {}
(-2,-4.5) node (11) [label=below:$11$] {}
(-2.5,-4.5)node (10) [label=below:$10$] {}
(-3.5,-4.5) node (9) [label=below:$9$] {}
(-4,-4.5) node (8) [label=below:$8$] {}
(-4.5,-4.5) node (7) [label=below:$7$] {}
(-5.5,-4.5) node (6) [label=below:$6$] {}
(-6,-4.5) node (5) [label=below:$5$] {}
(-6.5,-4.5) node (4) [label=below:$4$] {}
(-7.5,-4.5) node (3) [label=below:$3$] {}
(-8,-4.5) node (2) [label=below:$2$] {}
(-8.5,-4.5) node (1) [label=below:$1$] {};
\draw (1)--(28);\draw (2)--(28);\draw (3)--(28);\draw (4)--(29);\draw (5)--(29);\draw (6)--(29);\draw (7)--(30);\draw (8)--(30);\draw (9)--(30);\draw (10)--(31);\draw (11)--(31);\draw (12)--(31);\draw (13)--(32);\draw (14)--(32);\draw (15)--(32);\draw (16)--(33);\draw (17)--(33);\draw (18)--(33);
\draw (19)--(34);\draw (20)--(34);\draw (21)--(34);\draw (22)--(35);\draw (23)--(35);\draw (24)--(35);\draw (25)--(36);\draw (26)--(36);\draw (27)--(36);
\draw (28)--(37);\draw (29)--(37);\draw (30)--(37);\draw (31)--(38);\draw (32)--(38);\draw (33)--(38);\draw (34)--(39);\draw (35)--(39);\draw (36)--(39);
\draw (37)--(40);\draw (38)--(40);\draw (39)--(40);
\end{tikzpicture}
\caption{The tree $B(3,4)$}%
\end{figure}

Clearly, any Bethe tree is a generalized Bethe tree. Theorem \ref{spectrum}
immediately implies the following assertion:

\begin{corollary}
\label{bethe}Let $\alpha\in\left[  0,1\right)  $, and $\beta=1-\alpha$. For
any $j\in\left[  k\right]  $, let $T_{j}$ be the leading principal submatrix
of order $j\times j$ of the $k\times k$ symmetric tridiagonal matrix
\[
T_{k}=\left[
\begin{array}
[c]{ccccc}%
\alpha & \beta\sqrt{d} & 0 &  & 0\\
\beta\sqrt{d} & \alpha\left(  d+1\right)  & \beta\sqrt{d} &  & \\
& \ddots & \ddots & \ddots & \\
&  &  & \alpha\left(  d+1\right)  & \beta\sqrt{d}\\
0 &  & 0 & \beta\sqrt{d} & \alpha d
\end{array}
\right]  .
\]

(1) The spectrum of $A_{\alpha}(B(d,k))$ is the multiset union
\[
\mathrm{Spec}(T_{1})\cup\cdots\cup\mathrm{Spec}(T_{k})\text{;}%
\]

(2) The multiplicity of each eigenvalue of $T_{j}$ as an eigenvalue of
$A_{\alpha}(B(d,k))$ is $d^{k-j-1}(d-1)$ if \ $1\leq j\leq k-1$, and is $1$ if
$j=k$. If some eigenvalues obtained in different matrices are equal, their
multiplicities are added together;

(3) The largest eigenvalue of $T_{k}$ is the largest eigenvalue of $A_{\alpha
}(B(d,k))$.
\end{corollary}

\bigskip

\begin{proof}
[\textbf{Proof of Theorem \ref{t1}}]Let $\beta=1-\alpha$. As already
mentioned, each tree $T_{\Delta}$ with maximal degree $\Delta$ can be embedded
in a Bethe tree $B(\Delta-1,k)$ for sufficiently large $k.$ Hence,
\[
\rho(A_{\alpha}(T))\leq\rho(A_{\alpha}(B(\Delta-1,k))).
\]
On the other hand, Corollary \ref{bethe} implies that
\[
\rho(A_{\alpha}(B(\Delta-1,k)))=\rho\left(  T_{k}\right)  ,
\]
where
\[
T_{k}=\left[
\begin{array}
[c]{ccccc}%
\alpha & \beta\sqrt{\Delta-1} & 0 &  & 0\\
\beta\sqrt{\Delta-1} & \alpha\Delta & \beta\sqrt{\Delta-1} &  & \\
& \ddots & \ddots & \ddots & \\
&  &  & \alpha\Delta & \beta\sqrt{\Delta-1}\\
0 &  & 0 & \beta\sqrt{\Delta-1} & \alpha(\Delta-1)
\end{array}
\right]  .
\]
Since $T_{k}$ is an irreducible nonnegative matrix, with no equal row sums,
its largest eigenvalue is less than its largest row sum. That is,
\[
\rho(A_{\alpha}(B(\Delta-1,k)))<\alpha\Delta+2(1-\alpha)\sqrt{\Delta-1},
\]
completing the proof of Theorem \ref{t1}.
\end{proof}

\section{\label{ton}The maximum $A_{\alpha}$-spectral radius of a tree of
order $n$}

For the proof of Theorem \ref{t2} we need some general facts on the matrices
$A_{\alpha}.$ To begin with, if $G$ is a graph of order $n$ and $\mathbf{x}%
:=\left(  x_{1},\ldots,x_{n}\right)  $ is a real vector, the quadratic form
$\left\langle A_{\alpha}\left(  G\right)  \mathbf{x},\mathbf{x}\right\rangle $
can be represented as%
\[
\left\langle A_{\alpha}\left(  G\right)  \mathbf{x},\mathbf{x}\right\rangle
=\sum_{\left\{  u,v\right\}  \in E\left(  G\right)  }(\alpha x_{u}%
^{2}+2\left(  1-\alpha\right)  x_{u}x_{v}+\alpha x_{v}^{2}).
\]

The main tool in the proofs of Theorems \ref{t2} and \ref{t3} is Proposition
15 of \cite{Nik16}, which reads as:

\begin{proposition}
\label{pro1}Let $\alpha\in\left[  0,1\right)  $, and $G$ be a graph with
$A_{\alpha}\left(  G\right)  =A_{\alpha}.$ Let $u,v,w\in V\left(  G\right)  $
and suppose that $\left\{  u,v\right\}  \in E\left(  G\right)  $ and $\left\{
u,w\right\}  \notin E\left(  G\right)  .$ Let $H$ be the graph obtained from
$G$ by deleting the edge $\left\{  u,v\right\}  $ and adding the edge
$\left\{  u,w\right\}  .$ If $\mathbf{x}:=\left(  x_{1},\ldots,x_{n}\right)  $
is a unit eigenvector to $\lambda\left(  A_{\alpha}\right)  $ such that
$x_{u}>0$ and
\[
\left\langle A_{\alpha}\left(  H\right)  \mathbf{x},\mathbf{x}\right\rangle
\geq\left\langle A_{\alpha}\mathbf{x},\mathbf{x}\right\rangle ,
\]
then $\lambda\left(  A_{\alpha}\left(  H\right)  \right)  >\lambda\left(
A_{\alpha}\right)  .$
\end{proposition}

Having Proposition \ref{pro1} in hand, the proof of Theorem \ref{t2} is now
straightforward:\bigskip

\begin{proof}
[\textbf{Proof of Theorem \ref{t2}}]Let $T$ be a tree of order $n$ with
maximum spectral radius among all trees of order $n$. If $\alpha=1,$ then
$A_{\alpha}\left(  T\right)  =D\left(  T\right)  ,$ so $\rho\left(  D\left(
T\right)  \right)  $ is equal to the maximal degree of $T$, which is $n-1$ if
and only if $T=K_{1,n-1}.$

Next, assume that $0\leq\alpha<1$, let $A_{\alpha}:=A_{\alpha}\left(
T\right)  $, and let $\left(  x_{1},\ldots,x_{n}\right)  $ be a nonnegative
unit eigenvector to $\rho\left(  A_{\alpha}\right)  .$ Since $T$ is connected,
$\left(  x_{1},\ldots,x_{n}\right)  $ is positive (see, e.g., Proposition 13
of \cite{Nik16}). Choose $w\in V\left(  T\right)  $ such that
\[
x_{w}=\max\left\{  x_{1},\ldots,x_{n}\right\}  ,
\]
and assume that $T\neq K_{1,n-1}.$ Hence, there is a vertex $u\in V\left(
T\right)  $ of degree $1$ that is not connected to $w$. Write $v$ for the
neighbor of $u,$ delete the edge $\left\{  u,v\right\}  $, and add the edge
$\left\{  u,w\right\}  .$ Clearly the resulting graph $H$ is also a tree. We
find that
\begin{align*}
\left\langle A_{\alpha}\left(  H\right)  \mathbf{x},\mathbf{x}\right\rangle
-\left\langle A_{\alpha}\mathbf{x},\mathbf{x}\right\rangle  &  =(\alpha
x_{u}^{2}+2\left(  1-\alpha\right)  x_{u}x_{w}+\alpha x_{w}^{2})-(\alpha
x_{u}^{2}+2\left(  1-\alpha\right)  x_{u}x_{v}+\alpha x_{v}^{2})\\
&  =2\left(  1-\alpha\right)  x_{u}\left(  x_{w}-x_{v}\right)  +\alpha\left(
x_{w}-x_{v}\right)  \left(  x_{w}+x_{v}\right)  \geq0.
\end{align*}
Thus, Proposition \ref{pro1} implies that $\lambda\left(  A_{\alpha}\left(
H\right)  \right)  >\lambda\left(  A_{\alpha}\right)  ,$ contradicting the
choice of $T$. Hence, $\rho\left(  A_{\alpha}\left(  T\right)  \right)  $ is
maximal if and only if $T=K_{1,n-1}.$

To complete the proof of Theorem \ref{t2}, Proposition 38 in \cite{Nik16}
yields
\[
\rho\left(  A_{\alpha}\left(  K_{1,n-1}\right)  \right)  =\frac{\alpha
n+\sqrt{\alpha^{2}n^{2}+4\left(  n-1\right)  \left(  1-2\alpha\right)  }}{2}.
\]

\end{proof}

\section{\label{cg}Connected graphs with minimal $A_{\alpha}$-spectral radius}

Our proof of Theorem \ref{t3} is somewhat involved and makes use of a few
known results. \begin{figure}[ptb]%
\begin{align*}
\begin{tikzpicture} \node at (-6.4,-1.7) {$C_n$}; \node at (-1.6,-1.4) {$Y_n \left( n>5\right) $}; \node at (2.5,-1.4) {$K_{1,4}$}; \node at (-7.6,-4.6) {$F_7$}; \node at (-2.6,-4.6) {$F_8$}; \node at (3.3,-4.6) {$F_9$}; \tikzstyle{every node}=[draw,circle,fill=black,minimum size=3pt, inner sep=0pt] \draw (-5.529180,0.705342) node(11) {} (-6.129180,1.141268) node(12) {} (-6.870820,1.141268) node(13) {} (-7.470820,0.705342) node(14) {} (-7.700000,0.000000) node(15) {} (-7.470820,-0.705342) node(16) {} (-6.870820,-1.141268) node(17) {} (-6.129180,-1.141268) node(18) {} (-3,-0.7) node (21){} (-3,0.7) node (22){} (-3,0) node (23){} (-0.3,-0.7) node (26){} (-0.3,0) node (27){} (-0.3,0.7) node (28){} (2.5,0) node (31){} (1.8,0) node (32){} (3.2,0) node (33){} (2.5,-0.7) node (34){}(2.5,0.7) node (35){} (-9,-4) node (41){} (-8.3,-4) node (42){} (-7.6,-4) node (43){} (-6.9,-4) node (44){} (-6.2,-4) node (45){} (-7.6,-3.3) node (46){} (-7.6,-2.6) node (47){} (-4.8,-4) node (51){} (-4.1,-4) node (52){} (-3.4,-4) node (53){} (-2.7,-4) node (54) {} (-2.0,-4) node (55){}(-1.3,-4) node (56){} (-0.6,-4) node (57){} (-2.7,-3.3) node (58){} (0.8,-4) node (61){} (1.5,-4) node (62){} (2.2,-4) node (63){} (2.9,-4) node (64){} (3.6,-4) node (65){} (4.3,-4) node (66){} (5,-4) node (67){} (5.7,-4) node (68){} (2.2,-3.3) node (69){}; \tikzstyle{every node}=[draw,circle,fill=black,minimum size=1pt, inner sep=0pt] \draw (11)--(12); \draw (12)--(13); \draw (13)--(14);\draw (14)--(15); \draw (15)--(16); \draw (16)--(17); \draw (17)--(18); \draw (21)--(22); \draw (26)--(28);\draw (23)--(-2.3,0); \draw (-1,0)--(27); \draw (-1.9,0) node(26){} (-1.65,0)node(26){} (-1.4,0) node(26){} (-5.529180,-0.705342) node(19){} (-5.300000,-0.000000) node(10){}; \draw (32)--(33); \draw (34)--(35); \draw (41)--(42); \draw (42)--(43); \draw (43)--(44);\draw (44)--(45); \draw (46)--(43); \draw (47)--(46); \draw (51)--(57); \draw (58)--(54); \draw (61)--(68); \draw (69)--(63); \end{tikzpicture}
\end{align*}
\caption{The Smith graphs}%
\end{figure}
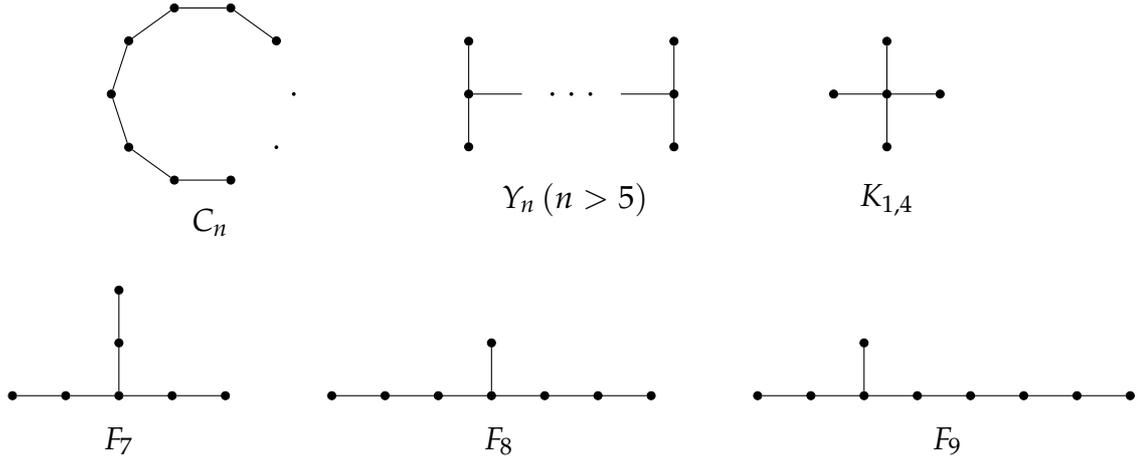

For a start, let us make two remarks on the graphs in Fig. 3. First, the
spectral radius of each of these graphs is precisely $2$; second, the
subscript in their notation stands for their order. The family of these graphs
was first outlined by J.H. Smith in \cite{Smi70}, who showed that any
connected graph with spectral radius at most $2$ is an induced subgraph of
some of them (see also \cite{CRS10}, p. 92). This fact is crucial for our
proof of Theorem \ref{t3}.\medskip

Next, we state a corollary of Proposition 3 of \cite{Nik16}:

\begin{proposition}
\label{pro2}If $G$ is a graph with maximal degree $\Delta$ and $\alpha
\in\left[  0,1\right]  ,$ then
\begin{equation}
\rho\left(  A\left(  G\right)  \right)  \leq\rho\left(  A_{\alpha}\left(
G\right)  \right)  \leq\Delta. \label{bo}%
\end{equation}
If $\rho\left(  A_{\alpha}\left(  G\right)  \right)  =\Delta$, then either
$\alpha=1$, or $G$ is regular.
\end{proposition}

We need also a simple property of the entries of an eigenvector to
$\rho\left(  A_{\alpha}\left(  P_{n}\right)  \right)  $, which seems natural,
but is not obvious:

\begin{proposition}
\label{pro3}Let $\alpha\in\left[  0,1\right)  $, and let $\left(  x_{1}%
,\ldots,x_{n}\right)  $ be a unit nonnegative eigenvector to $\rho\left(
A_{\alpha}\left(  P_{n}\right)  \right)  $. If $1\leq i\leq\left\lceil
n/2\right\rceil -1$, then%
\begin{equation}
x_{i}<x_{i+1}. \label{in1}%
\end{equation}
If $n$ is even, then
\begin{equation}
x_{n/2}=x_{n/2+1}. \label{eq1}%
\end{equation}

\end{proposition}

\begin{proof}
The Perron-Frobenius theory of nonnegative matrices implies that $\left(
x_{1},\ldots,x_{n}\right)  $ is unique and positive. Since $P_{n}$ is
symmetric about its center, it follows that
\[
x_{i}=x_{n-i+1}.
\]
for any $i\in\left[  \left\lfloor n/2\right\rfloor \right]  $. Hence, if $n$
is even, then $x_{n/2}=x_{n/2+1}.$

Set $\lambda:=\rho\left(  A_{\alpha}\left(  P_{n}\right)  \right)  $;
Proposition \ref{pro2} implies that $\lambda<2$. If $1<i\leq\left\lceil
n/2\right\rceil -1$, the eigenequations for $\lambda$ and $\left(
x_{1},\ldots,x_{n}\right)  $ read as
\[
\lambda x_{i}=2\alpha x_{i}+\left(  1-\alpha\right)  x_{i+1}+\left(
1-\alpha\right)  x_{i-1}\text{;}%
\]
hence,%
\begin{equation}
\left(  \lambda-2\alpha\right)  x_{i}=\left(  1-\alpha\right)  x_{i+1}+\left(
1-\alpha\right)  x_{i-1}. \label{eeq}%
\end{equation}
If $n$ is even and $i=n/2,$ we find that
\[
\left(  \lambda-2\alpha\right)  x_{n/2}=\left(  1-\alpha\right)
x_{n/2+1}+\left(  1-\alpha\right)  x_{n/2-1}.
\]
Hence, in view of (\ref{eq1}), we get
\[
x_{n/2}=\frac{1-\alpha}{\lambda-1-\alpha}x_{n/2-1}>x_{n/2-1}.
\]
If $n$ is odd, we find that
\[
\left(  \lambda-2\alpha\right)  x_{\left\lceil n/2\right\rceil }=\left(
1-\alpha\right)  x_{\left\lfloor n/2\right\rfloor }+\left(  1-\alpha\right)
x_{\left\lceil n/2\right\rceil -1}=2\left(  1-\alpha\right)  x_{\left\lfloor
n/2\right\rfloor }\text{;}%
\]
consequently,
\[
x_{\left\lfloor n/2\right\rfloor }<x_{\left\lceil n/2\right\rceil }.
\]
Now, we conclude the proof of (\ref{in1}) by induction on the difference
$k=i-\left\lceil n/2\right\rceil +1$. Up to this moment we have proved
(\ref{in1}) for $k=1$. Assume that $k>1$, and that inequality (\ref{in1})
holds for $k^{\prime}=k-1$. This assumption, together with (\ref{eeq}),
implies that
\[
\left(  \lambda-2\alpha\right)  x_{i+1}=\left(  1-\alpha\right)
x_{i+2}+\left(  1-\alpha\right)  x_{i}>\left(  1-\alpha\right)  x_{i+1}%
+\left(  1-\alpha\right)  x_{i},
\]
so%
\[
x_{i+1}>\frac{1-\alpha}{\lambda-1-\alpha}x_{i}>x_{i},
\]
completing the induction step and the proof of (\ref{in1}).
\end{proof}

Armed with these results, we are ready to prove Theorem \ref{t3}:\medskip

\begin{proof}
[\textbf{Proof of Theorem \ref{t3}}]Let $\alpha\in\left[  0,1\right]  $, and
let $G$ be a connected graph of order $n$ such that $\rho\left(  A_{\alpha
}\left(  G\right)  \right)  $ is minimal. Evidently $G$ is a tree, for
otherwise we can remove some edge of $G$, thereby diminishing $\rho\left(
A_{\alpha}\left(  G\right)  \right)  $, contrary to our choice. We assume that
$n\geq5,$ since the only two trees of order $4$ are $P_{4}$ and $K_{1,3}$, and
Theorem \ref{t2} implies that $\rho\left(  A_{\alpha}\left(  P_{4}\right)
\right)  <\rho\left(  A_{\alpha}\left(  K_{1,3}\right)  \right)  $.

For a start, let us note that if $\alpha=1,$ then
\[
\rho\left(  A_{\alpha}\left(  G\right)  \right)  =\rho\left(  D\left(
G\right)  \right)  =\Delta.
\]
Hence, $G=P_{n}$, since $P_{n}$ is the only tree with maximal degree equal to
$2$.

Next, assume that $0\leq\alpha<1$; hence, Proposition \ref{pro2} implies that
\[
\rho\left(  A\left(  G\right)  \right)  \leq\rho\left(  A_{\alpha}\left(
G\right)  \right)  \leq\rho\left(  A_{\alpha}\left(  P_{n}\right)  \right)
<2.
\]
Assume for a contradiction that $G\neq P_{n}.$ Thus, Smith's result implies
that $G$ is a proper induced subgraph of $Y_{n},F_{7},F_{8},$ or $F_{9}$. A
brief inspection of these graphs shows that only two cases are possible:

(a) $n\geq5$ and $G$ is a path $\left(  v_{2},\ldots,v_{n}\right)  $ with an
additional vertex $v_{1}$ joined to $v_{3}$;

(b) $n\geq6$ and $G$ is a path $\left(  v_{2},\ldots,v_{n}\right)  $ with an
additional vertex $v_{1}$ joined to $v_{4}$.

In either case, we apply Proposition \ref{pro1} to complete the proof.

Let $n\geq5$, let $\left(  v_{1},\ldots,v_{n}\right)  $ be a path $P_{n}$ of
order $n$, and let $\mathbf{x}:=\left(  x_{1},\ldots,x_{n}\right)  $ be a
positive unit eigenvector to $\rho\left(  A_{\alpha}\left(  P_{n}\right)
\right)  $. Delete the edge $\left\{  v_{1},v_{2}\right\}  $ and add the edge
$\left\{  v_{1},v_{3}\right\}  $, thereby obtaining the graph $G$ of case (a).
Since Proposition \ref{pro3} implies that $x_{2}<x_{3}$, we find that%
\begin{align*}
\left\langle A_{\alpha}\left(  G\right)  \mathbf{x},\mathbf{x}\right\rangle
-\left\langle A_{\alpha}\left(  P_{n}\right)  \mathbf{x},\mathbf{x}%
\right\rangle  &  =(\alpha x_{1}^{2}+2\left(  1-\alpha\right)  x_{1}%
x_{3}+\alpha x_{3}^{2})-(\alpha x_{1}^{2}+2\left(  1-\alpha\right)  x_{1}%
x_{2}+\alpha x_{2}^{2})\\
&  =2\left(  1-\alpha\right)  x_{1}\left(  x_{3}-x_{2}\right)  +\alpha\left(
x_{3}-x_{2}\right)  \left(  x_{3}+x_{2}\right)  >0.
\end{align*}
Therefore,
\[
\rho\left(  A_{\alpha}\left(  G\right)  \right)  \geq\left\langle A_{\alpha
}\left(  G\right)  \mathbf{x},\mathbf{x}\right\rangle >\left\langle A_{\alpha
}\left(  P_{n}\right)  \mathbf{x},\mathbf{x}\right\rangle =\rho\left(
A_{\alpha}\left(  P_{n}\right)  \right)  ,
\]
contradicting the choice of $G$; hence, $G=P_{n}$.

The proof of case (b) is carried out by a similar argument and is omitted.
\end{proof}

\bigskip

\section{\label{Sb}A few bounds on the spectral radius of $A_{\alpha}$}

In this section, we give upper and lower bounds for the spectral radii of the
matrices $A_{\alpha}\left(  G\right)  $ for any graph $G$. These bounds are
sufficiently good to deduce tight estimates of the spectral radii of
$A_{\alpha}(P_{n})$ and $A_{\alpha}\left(  B\left(  d,k\right)  \right)  $. \

We shall use Weyl's inequalities for eigenvalues of Hermitian matrices (see,
e.g. \cite{HoJo85}, p. 181). The conditions for equality in Weyl's
inequalities were first established by So in \cite{So94}$.$ For convenience we
state below the complete theorem of Weyl and So:\medskip

\textbf{Theorem WS} \emph{Let }$A$\emph{\ and }$B$\emph{\ be Hermitian
matrices of order }$n,$\emph{\ and let }$1\leq i\leq n$\emph{\ and }$1\leq
j\leq n.$\emph{\ Then}{%
\begin{align}
\lambda_{i}(A)+\lambda_{j}(B)  &  \leq\lambda_{i+j-n}(A+B),\text{if }i+j\geq
n+1,\label{Wein1}\\
\lambda_{i}(A)+\lambda_{j}(B)  &  \geq\lambda_{i+j-1}(A+B),\text{if }i+j\leq
n+1. \label{Wein2}%
\end{align}
}\emph{In either of these inequalities equality holds if and only if there
exists a nonzero }$n$\emph{-vector that is an eigenvector to each of the three
eigenvalues involved. }\medskip

A simplified version of (\ref{Wein1}) and (\ref{Wein2}) gives%
\begin{equation}
\lambda_{k}\left(  A\right)  +\lambda_{n}\left(  B\right)  \leq\lambda
_{k}\left(  A+B\right)  \leq\lambda_{k}\left(  A\right)  +\lambda_{1}\left(
B\right)  . \label{Wes}%
\end{equation}

Inequalities (\ref{Wes}), together with the basic identity%
\[
A_{\alpha}\left(  G\right)  -A_{\beta}\left(  G\right)  =\left(  \alpha
-\beta\right)  L\left(  G\right)
\]
were used in \cite{Nik16} to establish a number of bounds on the eigenvalues
of $A_{\alpha}\left(  G\right)  $. The same simple ideas can be used for
further refinement, as shown in the following two propositions.

\begin{proposition}
\label{ub}Let $G$ be a graph with $A\left(  G\right)  =A,$ $Q\left(  G\right)
=Q$, and maximum degree $\Delta.$

(i) If $0\leq\alpha\leq1/2,$ then%
\begin{equation}
\rho\left(  A_{\alpha}\left(  G\right)  \right)  \leq\alpha\rho\left(
Q\right)  +\left(  1-2\alpha\right)  \rho\left(  A\right)  . \label{ub1}%
\end{equation}
If $G$ is connected and irregular, equality holds in (\ref{ub1}) if and only
if $\alpha=0$ or $\alpha=1/2.$

(ii) If $1/2\leq\alpha\leq1,$ then%
\begin{equation}
\rho\left(  A_{\alpha}\left(  G\right)  \right)  \leq\left(  1-\alpha\right)
\rho\left(  Q\right)  +\left(  2\alpha-1\right)  \Delta. \label{ub2}%
\end{equation}
If $G$ is connected and irregular, equality holds in (\ref{ub2}) if and only
if $\alpha=1/2$ or $\alpha=1.$
\end{proposition}

\begin{proof}
Set $D:=D\left(  G\right)  $. To prove \emph{(i)} note that
\[
A_{\alpha}\left(  G\right)  =\alpha D+\left(  1-\alpha\right)  A=\alpha
Q+\left(  1-2\alpha\right)  A\text{;}%
\]
hence inequality (\ref{ub1}) follows by Weyl's inequality (\ref{Wes}).

Let $G$ be a connected and irregular graph, and assume for a contradiction
that $0<\alpha<1/2$ and equality holds in (\ref{ub1}). The condition for
equality in (\ref{Wes}) implies that $\rho\left(  Q\right)  $ and $\rho\left(
A\right)  $ have a common eigenvector $\mathbf{x}$; since $G$ is connected,
$\mathbf{x}$ has no zero entries. Clearly $\mathbf{x}$ is an eigenvector to
$D$, that is to say, there is some $\lambda$ such that $D\mathbf{x}%
=\lambda\mathbf{x}$. Therefore the degrees of $G$ are equal, contradicting the
premise that $G$ is irregular. Hence, if $G$ is connected and irregular, and
equality holds in (\ref{ub1}), then $\alpha=0$ or $\alpha=1/2.$ The converse
of this statement is obvious.

To prove \emph{(ii) }note that%
\[
A_{\alpha}=\alpha D+\left(  1-\alpha\right)  \left(  Q-D\right)  =\left(
1-\alpha\right)  Q+\left(  2\alpha-1\right)  D\text{;}%
\]
hence, inequality (\ref{ub2}) follows by Weyl's inequality (\ref{Wes}). The
condition for equality can be proved as in clause \emph{(i)}, so we omit it.
\end{proof}

\medskip

Next, we state another basic identity involving the matrices $A_{\alpha
}\left(  G\right)  $%
\begin{equation}
A_{\alpha}\left(  G\right)  +A_{1-\alpha}\left(  G\right)  =Q\left(  G\right)
. \label{id}%
\end{equation}

Coupled with Theorem WS, inequality (\ref{id}) gives the following result:

\begin{lemma}
\label{le1}If $G$ is a graph and $\alpha\in\left[  0,1\right]  $, then%
\begin{equation}
\rho\left(  A_{\alpha}\left(  G\right)  \right)  +\rho\left(  A_{1-\alpha
}\left(  G\right)  \right)  \geq\rho\left(  Q\left(  G\right)  \right)  .
\label{in}%
\end{equation}
If $G$ is regular, then equality holds in (\ref{in}) for any $\alpha\in\left[
0,1\right]  $. If $G$ is connected and equality holds in (\ref{in}) for some
$\alpha\neq1/2$, then $G$ is regular.
\end{lemma}

\begin{proof}
Inequality (\ref{in}) follows by applying Weyl's inequality (\ref{Wes}) to
identity (\ref{id}). If $G$ is a $d$-regular graph, then
\[
\rho\left(  A_{\alpha}\right)  =\alpha d+\left(  1-\alpha\right)  d=d,
\]
so equality holds in (\ref{in}) for any $\alpha\in\left[  0,1\right]  $.

Now let $G$ be connected and let
\[
\rho\left(  A_{\alpha}\left(  G\right)  \right)  +\rho\left(  A_{1-\alpha
}\left(  G\right)  \right)  =\rho\left(  Q\left(  G\right)  \right)
\]
for some $\alpha\neq1/2$. The conditions for equality in Weyl's inequalities
imply that $\rho\left(  Q\left(  G\right)  \right)  $, $\rho\left(  A_{\alpha
}\left(  G\right)  \right)  $, and $\rho\left(  A_{1-\alpha}\left(  G\right)
\right)  $ have a common eigenvector $\mathbf{x}=\left(  x_{1},\ldots
,x_{n}\right)  $. Hence
\[
\alpha\left(  D\left(  G\right)  \right)  \mathbf{x}+\left(  1-\alpha\right)
A\left(  G\right)  \mathbf{x}=\rho\left(  A_{\alpha}\left(  G\right)  \right)
\mathbf{x}%
\]
and%
\[
D\left(  G\right)  \mathbf{x}+A\left(  G\right)  \mathbf{x}=\rho\left(
Q\left(  G\right)  \right)  \mathbf{x}.
\]
These equalities lead to
\[
\left(  2\alpha-1\right)  D\left(  G\right)  \mathbf{x}=\left(  \rho\left(
A_{\alpha}\left(  G\right)  \right)  +\left(  \alpha-1\right)  \rho\left(
Q\left(  G\right)  \right)  \right)  \mathbf{x}\text{.}%
\]
Hence for any vertex $v\in V\left(  G\right)  $, we see that
\[
\left(  2\alpha-1\right)  d\left(  v\right)  x_{v}=\left(  \rho\left(
A_{\alpha}\left(  G\right)  \right)  +\left(  \alpha-1\right)  \rho\left(
Q\left(  G\right)  \right)  \right)  x_{v}.
\]
Since $G$ is connected, $\mathbf{x}$ has no zero entries; in addition,
$\left(  2\alpha-1\right)  \neq0$. Therefore, $G$ is regular, completing the
proof of the lemma.
\end{proof}

Obviously, identity (\ref{id}) can be applied to transform lower bounds on
$\rho\left(  A_{\alpha}\left(  G\right)  \right)  $ into upper ones, and vice
versa. Below, we shall combine it with Proposition \ref{ub} to produce lower
bounds on $\rho\left(  A_{\alpha}\left(  G\right)  \right)  $.

\begin{proposition}
\label{lb}Let $G$ be a graph with $A\left(  G\right)  =A,$ $Q\left(  G\right)
=Q$, and maximum degree $\Delta.$

(i) If $0\leq\alpha\leq1/2,$ then%
\begin{equation}
\rho\left(  A_{\alpha}\left(  G\right)  \right)  \geq\left(  1-\alpha\right)
\rho\left(  Q\right)  +\left(  2\alpha-1\right)  \Delta\text{;} \label{lb1}%
\end{equation}
If $G$ is connected and irregular, equality holds in (\ref{lb1}) if and only
if $\alpha=1/2.$

(ii) If $1/2\leq\alpha\leq1,$ then%
\begin{equation}
\rho\left(  A_{\alpha}\left(  G\right)  \right)  \geq\alpha\rho\left(
Q\right)  +\left(  1-2\alpha\right)  \rho\left(  A\right)  . \label{lb2}%
\end{equation}
If $G$ is connected and irregular, equality holds in (\ref{lb2}) if and only
if $\alpha=1/2.$
\end{proposition}

\begin{proof}
For a start, note that identity (\ref{id}), together with inequality
(\ref{Wes}), gives \emph{ }%
\begin{equation}
\rho\left(  A_{\alpha}\left(  G\right)  \right)  +\rho\left(  A_{1-\alpha
}\left(  G\right)  \right)  \geq\rho\left(  Q\right)  . \label{Wb}%
\end{equation}
To prove \emph{(i)} note that if $0\leq\alpha\leq1/2,$ then $1/2\leq
1-\alpha\leq1,$ so bound (\ref{ub2}) implies that%
\begin{equation}
\rho\left(  A_{1-\alpha}\left(  G\right)  \right)  \leq\alpha\rho\left(
Q\right)  +\left(  1-2\alpha\right)  \Delta\text{.} \label{ub21}%
\end{equation}
Hence, in view of (\ref{Wb}), we get%
\begin{align*}
\rho\left(  A_{\alpha}\left(  G\right)  \right)   &  \geq\rho\left(  Q\right)
-\rho\left(  A_{1-\alpha}\left(  G\right)  \right) \\
&  \geq\rho\left(  Q\right)  -\left(  1-\left(  1-\alpha\right)  \right)
\rho\left(  Q\right)  -\left(  2\left(  1-\alpha\right)  -1\right)  \Delta\\
&  =\left(  1-\alpha\right)  \rho\left(  Q\right)  +\left(  2\alpha-1\right)
\Delta\text{,}%
\end{align*}
proving (\ref{lb1}). Let $G$ be connected and irregular. If equality holds in
(\ref{lb1}), then
\[
\rho\left(  A_{\alpha}\left(  G\right)  \right)  +\rho\left(  A_{1-\alpha
}\left(  G\right)  \right)  =\rho\left(  Q\left(  G\right)  \right)  \text{,}%
\]
and Lemma \ref{le1} implies that $\alpha=1/2$. Clearly, if $\alpha=1/2,$ then
equality holds in (\ref{lb1}).

The proof of \emph{(ii)} follows the same idea. If $1/2\leq\alpha\leq1,$ then
\ $0\leq1-\alpha\leq1/2$, so bound (\ref{ub1}) yields
\begin{align*}
\rho\left(  A_{1-\alpha}\left(  G\right)  \right)   &  \leq\left(
1-\alpha\right)  \rho\left(  Q\right)  +\left(  1-2\left(  1-\alpha\right)
\right)  \rho\left(  A\right) \\
&  =\left(  1-\alpha\right)  \rho\left(  Q\right)  +\left(  2\alpha-1\right)
\rho\left(  A\right)  .
\end{align*}
Hence,%
\begin{align*}
\rho\left(  A_{a}\left(  G\right)  \right)   &  \geq\rho\left(  Q\right)
-\left(  \left(  1-\alpha\right)  \rho\left(  Q\right)  +\left(
2\alpha-1\right)  \rho\left(  A\right)  \right) \\
&  =\alpha\rho\left(  Q\right)  +\left(  1-2\alpha\right)  \rho\left(
A\right)  ,
\end{align*}
proving (\ref{lb2}). The condition for equality can be proved as in clause
\emph{(i)}.
\end{proof}

\medskip

Next, we apply Propositions \ref{ub} and \ref{lb} to estimate $\rho\left(
A_{\alpha}\left(  P_{n}\right)  \right)  $. Recall that
\[
\rho\left(  A\left(  P_{n}\right)  \right)  =2\cos\left(  \frac{\pi}%
{n+1}\right)  \text{ \ and \ \ \ }\rho\left(  Q\left(  P_{n}\right)  \right)
=2+2\cos\left(  \frac{\pi}{n}\right)  .
\]
Combining these facts with Proposition \ref{ub}, we get a tight upper bound on
$\rho\left(  A_{\alpha}\left(  P_{n}\right)  \right)  $:

\begin{corollary}
\label{ubc}The spectral radius of $A_{\alpha}\left(  P_{n}\right)  $
satisfies
\begin{equation}
\rho\left(  A_{\alpha}\left(  P_{n}\right)  \right)  \leq\left\{
\begin{array}
[c]{ll}%
2\alpha+2\left(  1-\alpha\right)  \cos\left(  \frac{\pi}{n+1}\right)  , &
\text{if }0\leq\alpha<1/2\text{;}\\
2\alpha+2\left(  1-\alpha\right)  \cos\left(  \frac{\pi}{n}\right)  , &
\text{if }1/2\leq\alpha\leq1\text{.}%
\end{array}
\right.  \label{up}%
\end{equation}
Equality holds if and only if $\alpha=0,$ $\alpha=1/2,$ or $\alpha=1$.
\end{corollary}

\begin{proof}
If $0\leq\alpha<1/2,$ then (\ref{ub1}) implies that%
\begin{align*}
\rho\left(  A_{\alpha}\left(  P_{n}\right)  \right)   &  \leq\alpha\rho\left(
Q\left(  P_{n}\right)  \right)  +\left(  1-2\alpha\right)  \rho\left(
A\left(  P_{n}\right)  \right) \\
&  =\alpha\left(  2+2\cos\left(  \frac{\pi}{n}\right)  \right)  +2\left(
1-2\alpha\right)  \cos\left(  \frac{\pi}{n+1}\right) \\
&  =2\alpha+2\left(  1-\alpha\right)  \cos\left(  \frac{\pi}{n+1}\right)
+2\alpha\left(  \cos\left(  \frac{\pi}{n}\right)  -\cos\left(  \frac{\pi}%
{n+1}\right)  \right) \\
&  \leq2\alpha+2\left(  1-\alpha\right)  \cos\left(  \frac{\pi}{n+1}\right)  .
\end{align*}
In this case, it is clear that if $\alpha=0$, then equality holds in
(\ref{up}). Conversely, if
\[
\rho\left(  A_{\alpha}\left(  P_{n}\right)  \right)  =2\alpha+2\left(
1-\alpha\right)  \cos\left(  \frac{\pi}{n+1}\right)  ,
\]
then Proposition \ref{ub}, \emph{(i)} implies that $\alpha=0.$

If $1/2\leq\alpha\leq1$, then (\ref{ub2}) implies that%
\begin{align*}
\rho\left(  A_{\alpha}\left(  G\right)  \right)   &  \leq\left(
1-\alpha\right)  \rho\left(  Q\left(  P_{n}\right)  \right)  +\left(
2\alpha-1\right)  \Delta\left(  P_{n}\right) \\
&  =\left(  1-\alpha\right)  \left(  2+2\cos\left(  \frac{\pi}{n}\right)
\right)  +2\left(  2\alpha-1\right) \\
&  =2\alpha+2\left(  1-\alpha\right)  \cos\left(  \frac{\pi}{n}\right)  .
\end{align*}
In this case, it is clear that if $\alpha=1/2$ or $\alpha=1$, then equality
holds in (\ref{up}). Conversely, if%
\[
\rho\left(  A_{\alpha}\left(  P_{n}\right)  \right)  =2\alpha+2\left(
1-\alpha\right)  \cos\left(  \frac{\pi}{n}\right)  ,
\]
then clause \emph{(ii) }of\emph{ }Proposition \ref{ub} implies that
$\alpha=1/2$ or $\alpha=1$.
\end{proof}

In a similar way, Proposition \ref{lb} implies a tight lower bound on
$\rho\left(  A_{\alpha}\left(  P_{n}\right)  \right)  :$

\begin{corollary}
\label{lbc}The spectral radius of $A_{\alpha}\left(  P_{n}\right)  $
satisfies
\begin{equation}
\rho\left(  A_{\alpha}\left(  P_{n}\right)  \right)  \geq\left\{
\begin{array}
[c]{ll}%
2\alpha+2\left(  1-\alpha\right)  \cos\left(  \frac{\pi}{n}\right)  , &
\text{if }0\leq\alpha\leq1/2\text{;}\\
2\alpha+2\alpha\cos\left(  \frac{\pi}{n}\right)  -2\left(  2\alpha-1\right)
\cos\left(  \frac{\pi}{n+1}\right)  , & \text{if }1/2<\alpha\leq1\text{.}%
\end{array}
\right.  \label{lp}%
\end{equation}
Equality holds if and only if $\alpha=1/2$.
\end{corollary}

\begin{proof}
If $0\leq\alpha\leq1/2,$ then inequality (\ref{lb1}) implies that
\begin{align*}
\rho\left(  A_{\alpha}\left(  G\right)  \right)   &  \geq\left(
1-\alpha\right)  \rho\left(  Q\left(  P_{n}\right)  \right)  +\left(
2\alpha-1\right)  \Delta\left(  P_{n}\right) \\
&  =\left(  1-\alpha\right)  \left(  2+2\cos\left(  \frac{\pi}{n}\right)
\right)  +2\left(  2\alpha-1\right) \\
&  =2\alpha+2\left(  1-\alpha\right)  2\cos\left(  \frac{\pi}{n}\right)  .
\end{align*}
If equality holds in (\ref{lp}), then the condition for equality in
(\ref{lb1}) implies that $\alpha=1/2.$

If $1/2<\alpha\leq1,$ then inequality (\ref{lb2}) implies that%
\begin{align*}
\rho\left(  A_{\alpha}\left(  G\right)  \right)   &  \geq\alpha\rho\left(
Q\left(  P_{n}\right)  \right)  +\left(  1-2\alpha\right)  \rho\left(
A\left(  P_{n}\right)  \right) \\
&  =\alpha\left(  2+2\cos\left(  \frac{\pi}{n}\right)  \right)  +2\left(
1-2\alpha\right)  \cos\left(  \frac{\pi}{n+1}\right) \\
&  \geq2\alpha+2\alpha\cos\left(  \frac{\pi}{n}\right)  -2\left(
2\alpha-1\right)  \cos\left(  \frac{\pi}{n+1}\right)  .
\end{align*}
It is not hard to see that in this case inequality (\ref{lp}) is always strict.
\end{proof}

Finally, we give tight bounds on the spectral radius of the Bethe tree
$B\left(  d,k\right)  .$

\begin{proposition}
If $0\leq\alpha\leq1$, then
\begin{equation}
\rho\left(  A_{\alpha}\left(  B\left(  d,k\right)  \right)  \right)
\leq\alpha\left(  d+1\right)  +2\left(  1-\alpha\right)  \sqrt{d}\cos\left(
\frac{\pi}{k+1}\right)  \label{ubb}%
\end{equation}
and%
\begin{equation}
\rho\left(  A_{\alpha}\left(  B\left(  d,k\right)  \right)  \right)
>\alpha\left(  d+1\right)  +2\left(  1-\alpha\right)  \sqrt{d}\cos\left(
\frac{\pi}{k}\right)  -\frac{20\alpha\sqrt{d}}{k^{3}}. \label{lbb}%
\end{equation}

\end{proposition}

\begin{proof}
Let $\beta=1-\alpha$. Clause \emph{(3)} of Corollary \ref{bethe} implies that
the spectral radius of $A_{\alpha}(B(d,k))$ is the largest eigenvalue of the
$k\times k$ matrix
\[
T_{k}=\left[
\begin{array}
[c]{ccccc}%
\alpha & \beta\sqrt{d} & 0 &  & 0\\
\beta\sqrt{d} & (d+1)\alpha & \beta\sqrt{d} &  & \\
& \ddots & \ddots & \ddots & \\
&  &  & (d+1)\alpha & \beta\sqrt{d}\\
0 &  & 0 & \beta\sqrt{d} & d\alpha
\end{array}
\right]  .
\]
Clearly, we have
\[
T_{k}=\beta\sqrt{d}A(P_{k})+\left[
\begin{array}
[c]{ccccc}%
\alpha & 0 &  &  & 0\\
0 & (d+1)\alpha &  &  & \\
&  & \ddots &  & \\
&  &  & (d+1)\alpha & 0\\
0 &  &  & 0 & d\alpha
\end{array}
\right]  .
\]
Hence, inequality (\ref{Wes}) implies that
\begin{align*}
\rho(T_{k})  &  =\rho(A_{\alpha}(B(d,k)))\leq\beta\sqrt{d}\rho(A(P_{k}%
))+(d+1)\alpha\\
&  \leq(d+1)\alpha+2(1-\alpha)\sqrt{d}\cos\left(  \frac{\pi}{k+1}\right)  .
\end{align*}

To prove (\ref{lbb}), recall that in \cite{RoRo07}, Theorem 9, it was shown
that
\[
\rho\left(  Q\left(  B\left(  d,k\right)  \right)  \right)  =\left(
d+1\right)  +2\sqrt{d}\cos\left(  \frac{\pi}{k}\right)  .
\]
Now, applying (\ref{id}) and (\ref{ubb}), we get
\begin{align*}
\rho(A_{\alpha}(B(d,k)))  &  \geq\rho\left(  Q\left(  B\left(  d,k\right)
\right)  \right)  -\rho(A_{1-\alpha}(B(d,k)))\\
&  \geq\left(  d+1\right)  +2\sqrt{d}\cos\left(  \frac{\pi}{k}\right)
-2\alpha\sqrt{d}\cos\left(  \frac{\pi}{k+1}\right)  -\left(  1-\alpha\right)
\left(  d+1\right) \\
&  =\alpha\left(  d+1\right)  +2(1-\alpha)\sqrt{d}\cos\left(  \frac{\pi}%
{k}\right)  -2\alpha\sqrt{d}\left(  \cos\left(  \frac{\pi}{k+1}\right)
-\cos\left(  \frac{\pi}{k}\right)  \right)  .
\end{align*}
On the other hand, the Mean Value Theorem implies that
\[
\cos\left(  \frac{\pi}{k+1}\right)  -\cos\left(  \frac{\pi}{k}\right)
=\frac{\pi}{k\left(  k+1\right)  }\sin\left(  \theta\right)
\]
for some $\theta\in\left[  \pi/\left(  k+1\right)  ,\pi/k\right]  $. Since
$\sin\left(  \theta\right)  \leq\theta\leq\pi/k,$ we get
\[
\cos\left(  \frac{\pi}{k+1}\right)  -\cos\left(  \frac{\pi}{k}\right)
=\frac{\pi}{k\left(  k+1\right)  }\sin\left(  \theta\right)  <\frac{\pi^{2}%
}{k^{2}\left(  k+1\right)  }<\frac{10}{k^{3}},
\]
completing the proof of (\ref{lbb}).
\end{proof}

\bigskip

\textbf{Acknowledgement. }Part of this work has been accomplished during the
2016 COMCA conference, hosted by the Universidad Cat\'{o}lica del Norte,
Antofagasta, Chile. The first author is grateful to his hosts for the
wonderful experience.

\bigskip

E-mails:

\ \ \ Vladimir Nikiforov - \textit{vnikifrv@memphis.edu}

\ \ \ Germain Past\'{e}n - \textit{germain.pasten@ucn.cl}

\ \ \ Oscar Rojo - \textit{orojo@ucn.cl}

\ \ \ Ricardo L. Soto - \textit{rsoto@ucn.cl}

\end{document}